\documentclass[11pt]{amsart}

\usepackage{amssymb,amsmath,amsthm}
\usepackage{mathrsfs,dsfont,a4wide}
\usepackage{verbatim}
\usepackage{amsfonts,amsmath,amssymb}
\usepackage[dvips]{graphicx}
\usepackage{psfrag}
\usepackage{epsfig}
\theoremstyle{plain}
\newtheorem{theorem}{Theorem}[section]
\newtheorem{lemma}[theorem]{Lemma}
\newtheorem{proposition}[theorem]{Proposition}

\newtheorem{remark}[theorem]{Remark}
\newtheorem{definition}[theorem]{Definition}
\theoremstyle{definition}
\theoremstyle{remark}
\numberwithin{equation}{section}

\newcommand{\cs}{{\mathcal C}}

\newcommand{\caL}{{\mathcal L}}

\newcommand{\ms}{{\mathcal M}}

\newcommand{\U}{{\mathcal U}}

\newcommand{\R}{{\mathbb R}}
\newcommand{\C}{{\mathbb C}}

\newcommand{\M}{{\mathbb M}^{2\times 2}}

\newcommand{\Om}{\Omega}

\newcommand{\no}{\noindent}
\newcommand{\diag}{{\rm diag}}
\newcommand{\Adj}{{\rm Adj}}
\newcommand{\dsp}{\displaystyle}

\newcommand{\Md}{\mathbb M^{2\times 2}}
\newcommand{\Msim}{\mathbb M_{sym}^{2\times 2}}

\newcommand{\Msdet}{ SL_{sym}(2)}
\renewcommand{\div}{\rm div}
\newcommand{\tr}{\text{tr}}

\renewcommand{\no}{\noindent}
\newcommand{\msl}{\ms(\lambda,\Om)}
\newcommand{\mslp}{\ms(\lambda^\prime,\Om)}

\newcommand{\res}{\mathop{\hbox{\vrule height 7pt width .5pt depth 0pt
\vrule height .5pt width 6pt depth 0pt}}\nolimits}

\title
[Gradient integrability and rigidity results for two-phase conductivities]
{Gradient integrability and rigidity results for two-phase conductivities  in dimension two}

\author[V. Nesi]
{Vincenzo Nesi}
\address[V.  Nesi]{Dipartimento di Matematica ``G. Castelnuovo", Sapienza, Universit\'a di Roma,
Piazzale A. Moro 2, 00185 Roma, Italy}
\email[V. Nesi]{nesi@mat.uniroma1.it}
\author[M. Palombaro]
{Mariapia Palombaro}
\address[M. Palombaro]{Centro De Giorgi, Scuola Normale Superiore, Piazza dei Cavalieri 3, 56126 Pisa, Italy} 
\email[M. Palombaro]{mariapia.palombaro@sns.it}
\author[M. Ponsiglione]
{Marcello Ponsiglione}
\address[M. Ponsiglione]{Dipartimento di Matematica ``G. Castelnuovo", Sapienza, Universit\'a di Roma,
Piazzale A. Moro 2, 00185 Roma, Italy} 
\email[M. Ponsiglione]{ponsigli@mat.uniroma1.it}

\begin{document}
\vskip .2truecm
\begin{abstract}
\small{
This paper deals with higher gradient integrability for $\sigma$-harmonic functions $u$ with discontinuous coefficients $\sigma$, i.e. weak solutions of $\div(\sigma \nabla u) = 0$.
 We focus on  two-phase conductivities $\sigma:\Om\subset\R^2\mapsto \{\sigma_1,\sigma_2\}\subset \Md$, and  study the higher integrability of the corresponding gradient field $|\nabla u|$.
The gradient field  and its integrability clearly depend on the geometry, i.e., on the phases arrangement described by the sets $E_i=\sigma^{-1}(\sigma_i)$. We find the optimal integrability exponent of the  gradient field corresponding to any  pair $\{\sigma_1,\sigma_2\}$ of positive definite matrices, i.e., the worst among all  possible microgeometries. We also  show that it is attained by  so-called exact solutions of the corresponding PDE.
Furthermore, among all two-phase conductivities with fixed ellipticity, we characterize those that correspond to the worse integrability.
\vskip.1truecm
\noindent Keywords: Beltrami system, quasiconformal mappings, elliptic equations, composites, gradient integrability. 

\no
\noindent 2000 Mathematics Subject Classification: 30C62, 35B27.}
\end{abstract}

\maketitle
{\small \tableofcontents}

\section{Introduction}
Let $\Om$ be a bounded, open and simply connected subset of $\R^2$ with Lipschitz continuous boundary. We are interested in elliptic equations  in  divergence form  with $L^{\infty}$ coefficients, specifically,
\begin{equation}\label{elliptic}
\div(\sigma \nabla u) = 0 \quad \text{ in } \Om.
\end{equation}
Here  $\sigma$ is a matrix valued coefficient, referred to  as   {\it conductivity}, and
any weak solution $u\in H^1_{loc}(\Om)$ to the  equation is called a {\it $\sigma$-harmonic} function. 
The case of discontinuous conductivities $\sigma$ is particularly relevant in the context of non homogeneous and composite materials. 
With this motivation, we only assume ellipticity. 
Denote by $\Md$ the space of real $2\times 2$ matrices and by
$\Msim$ the subspace of symmetric matrices.
\begin{definition}\label{ell-TM}
Let $\lambda\in (0,1]$.  We say that $\sigma\in L^\infty (\Om;\Md)$ belongs to  the class $\ms(\lambda,\Om)$ if it satisfies the following
uniform bounds
\begin{align}\label{ellipticity-sigma1}
\sigma \xi \cdot \xi  & \geq  \lambda |\xi|^2   & \text{ for every } \xi\in\R^2 \text{ and for a.e. }x\in\Om    \,,\\
\label{ellipticity-sigma2}
\sigma^{-1} \xi \cdot \xi   & \geq \lambda |\xi|^2  &\text{ for every } \xi\in\R^2 \text{ and for a.e. }x\in\Om \,,
\end{align}
Moreover, we denote by $\ms_{sym}\left(\lambda,\Om\right)$ the set 
$ \Msim \cap \ms\left(\lambda,\Om\right)$.
\end{definition}
\no
The reader may wonder why to use the notion of ellipticity given in Definition \ref{ell-TM}. 
For an explanation related to its relationship with $H$-convergence we refer the reader to 
\cite{AN}.

It is well known that the gradient of $\sigma$-harmonic functions locally belongs to some $L^p$ with $p>2$. The main goal of this paper is to explore this issue, focusing on two-phase conductivities   $\sigma:\Om\mapsto  \{\sigma_1, \sigma_2\}\subset \ms$. We will review known results and prove some new one.

Any $\sigma$-harmonic function $u$ can be seen as the real part of a complex map  
$\dsp f:\Om\mapsto\C$ which is a $H^1_{loc}$ solution to the {\it Beltrami equation}
\begin{equation}\label{beltrami}
 f_{\bar{z}}=\mu \, f_{z}+ \nu \, \overline{ f_{z}},\quad \text{ in }\Om\,,
 \end{equation}
 where the so called complex dilatations $\mu$ and $\nu$, both belonging to $L^{\infty}(\Om,\mathbb C)$, are given by 
 
 \begin{equation}\label{mu-nu(sigma)}
 \mu=\frac{\sigma_{22}-\sigma_{11}-i(\sigma_{12}+\sigma_{21})}{1+\tr \, \sigma+\det\sigma}\,,
 \quad
 \nu=\frac{1-\det\sigma+i(\sigma_{12}-\sigma_{21})}{1+\tr \, \sigma+\det\sigma}\,,
\end{equation}

\no 
and satisfy the ellipticity condition 
\begin{equation}\label{ellipticity-munu}
\| |\mu|+|\nu| \|_{L^\infty}< 1 \,.
\end{equation}
Let us recall that weak solutions to \eqref{beltrami} are called quasiregular mappings. They are called quasiconformal if, in addition, they are injective.
The ellipticity \eqref{ellipticity-munu} can be expressed by
\begin{equation}\label{ellipticity-munu-K}
\| |\mu|+|\nu| \|_{L^\infty}\leq \frac{K-1}{K+1} \,,
\end{equation}
for some $K>1$. The corresponding solutions to \eqref{beltrami} are called $K$-quasiregular, and $K$-quasiconformal if, in addition, they are injective.
In 1994,  K. Astala \cite{A} proved one of the 
most important pending conjectures in the field, namely that planar $K$-
quasiregular mappings have Jacobian determinant in 
$L^{K/(K-1)}_{weak}$. 
Astala's work represented a benchmark for the issue of determining the optimal 
integrability exponent which was previously studied in  the work of 
Bojarski \cite{boj} and N. Meyers \cite{M}. 

Summarizing, to  any given $\sigma \in \ms\left(\lambda,\Om\right)$ one can associate a 
corresponding pair of complex dilations via \eqref{mu-nu(sigma)} and 
therefore, via the Beltrami equation \eqref{beltrami}  a quasiregular mapping.
Therefore, given $\lambda\in [0,1)$ and given 
$\sigma \in \ms\left(\lambda,\Om\right)$ one can find $K=K(\sigma)$ by using 
\eqref{mu-nu(sigma)} and \eqref{ellipticity-munu-K} in such a way that the 
$\sigma$-harmonic function  $u$, solution to \eqref{elliptic} is the real part of a $K$-quasiregular mapping.  
The Astala regularity results in this context reads as $|\nabla u |\in L^{p_K}_{weak}(\Om)$, where 
$p_K:={\frac{2 K}{K-1}}$.  

A more refined issue is to determine weighted estimates for the Jacobian determinant of a quasiconformal mapping.  A first result in this direction was given in  \cite{AN1}.  A much finer recent result, is given in \cite{AIPS},  see formula (1.6). Throughout the  present paper we focus on the simpler framework of $L^p$ spaces. 

The first question is to
determine the best possible (i.e. the minimal) constant $K(\sigma)$ such that if $u$ is $\sigma$-harmonic with $\sigma \in \ms\left(\lambda,\Om\right)$, then $u$ is the real part of a $K(\sigma)$-quasiregular mapping.  
Astala writes in his celebrated paper that his result implies sharp exponents of integrability for the gradient of solutions of planar elliptic pdes of the form \eqref{elliptic}, and he says: ``note that the dilation of $f$ and so necessarily the optimal integrability exponent depends in a complicated manner on all the entries of the matrix $\sigma$ rather than just on its ellipticity''.
Alessandrini and Nesi \cite{AN}, in the process of proving the $G$-stability of Beltrami equations, made a progress which can be found in  their Proposition 1.8. Let us rephrase it here. 
See also \cite{AM} for the estimate \eqref{bestboundsym}.
\begin{proposition}\label{ALN-bis}
Let $\lambda\in(0,1]$. Then 
\begin{equation}\label{bestbound1}
K_\lambda := \sup_{\sigma \in \ms\left(\lambda,\Om\right)} K(\sigma)=\frac{1+\sqrt{1-\lambda^2}}{\lambda}\,,
\end{equation}
\begin{equation}\label{bestboundsym}
K_\lambda^{sym} :=
\sup_{\sigma \in \ms_{sym}\left(\lambda,\Om\right)} K(\sigma)=\frac{ 1}{\lambda}\,.
\end{equation}
\end{proposition}
In Section \ref{sec:reformulation} we give a simpler and more geometrical proof 
of Proposition \ref{ALN-bis} based on the real formulation of the 
Beltrami equation (see Propositions \ref{ALN} and \ref{optimality}).
In \cite{AN}, pp. 63, the authors noticed that, the supremum in 
\eqref{bestbound1} is attained on specific non symmetric matrices.
As a  straightforward corollary, in  \cite{AN} the authors write the version of 
Astala's theorem which is adequate for matrices belonging to  
$\ms\left(\lambda,\Om\right)$ that we recall here in an informal way. 
Any $\sigma$-harmonic function with $\sigma\in \ms\left(\lambda,\Om\right)$ 
satisfies the property
$|\nabla u |\in L^{p_{K_\lambda}}_{weak}$, where 
$K_{\lambda}$ is given by \eqref{bestbound1} and $p_{K_\lambda}:={\frac{2 K_{\lambda}}{K_{\lambda}-1}}$. 
This has to be compared with the version that holds true assuming a priori 
that  $\sigma\in  \ms_{sym}\left(\lambda,\Om\right)$. 
In that case  $K_{\lambda}$ can be replaced by 
$K_\lambda^{sym}$ defined in \eqref{bestboundsym}.  
Optimality in the latter case was proved by Leonetti and Nesi \cite{LN} which began their 
work using the bound \eqref{bestboundsym} which had been already observed 
in Alessandrini and Magnanini \cite{AM}.
Optimality means that
 there exists 
$\sigma \in \ms_{sym}(\lambda,\Om)$  for which the 
estimate $|\nabla u|\in L^{p_{K_\lambda}}_{weak}$ is sharp.

Later there has been a number of increasingly refined results showing optimality of Astala's theorem for a different class of symmetric matrices $\sigma$. Specifically Faraco \cite{F} treats the case of two isotropic materials, i.e. when $\sigma$ takes values only in the set of two matrices of the form $\{K I, \frac{1}{K} I \}$, with $I$ the identity matrix, which was originally conjectured to be optimal for the exponent $\frac{2 K}{K-1}$ by Milton \cite{M-threshold}.  
In a further advance a more refined version was given in \cite{AFS}, where the authors proved optimality in the stronger sense of exact solutions.

However the original question implicitly raised by Astala was apparently forgotten. 
In this paper we go back to that and we prove optimality for a generic two-phase matrix field $\sigma \in \ms\left(\lambda,\Om\right)$. 
To describe our approach  let us first recall that when $\sigma$ is smooth, the corresponding $\sigma$-harmonic function is necessarily smooth and hence with bounded gradient. So the issue of higher exponent of integrability is really related to discontinuous coefficients. The simplest class of examples is when one has a conductivity taking only two values. We therefore ask the following questions. Given two positive definite matrices, 
$\sigma_1$ and $\sigma_2$, consider the class of matrices 
$\sigma\in \ms\left(\lambda,\Om\right)$ of the special form  
$\sigma(x)= \sigma_1\chi_{E_1} +\sigma_2\chi_{E_2}$, 
where $\{E_1, \, E_2\}$ is a measurable partition of $\Om$ and $\chi_{E_i}$ denotes  the characteristic function of the set $E_i$. In the jargon of composite materials this is called a  two-phase composite. What is the best possible information one can extrapolate from Astala's Theorem? As already explained, to the ellipticity $\lambda$ of $\sigma$ there corresponds  a suitable constant $K(\sigma)$ in the Beltrami equation. 
We are naturally led to  the following related question: given $\mu,\nu\in L^\infty(\Om;\C)$ satisfying \eqref{ellipticity-munu-K} with $K(\mu,\nu)>1$, 
is it possible to transform $\mu$ and $\nu$, by a suitable change of variables, specifically, by  affine  transformations, in order to decrease $K$ and thus gain a better integrability for the solution of the transformed Beltrami equation?  The key observation here is that the summability  of  solutions of the Beltrami equation  is invariant under such transformations, while   $K(\mu,\nu)$ is not. It is then well defined the minimal Beltrami constant $K^{min}$ attainable under such transformations. In Proposition \ref{sc}  we find an explicit formula for such $K^{min}$ in terms of all the entries of $\sigma_1$ and $\sigma_2$. Moreover, $K^{min}$ gives a sharp measure of the integrability properties of solutions to \eqref{elliptic}. This is stated in Theorem \ref{thm:sigma-thm}, which, for the reader's convenience,  we reformulate 
here in a more informal way.

 \begin{theorem}\label{thm:sigma-thm0}
i) Let $\sigma\in \ms(\lambda,\Om)$ with $\sigma\in\{\sigma_1,\sigma_2\}$.
Every $\sigma$-harmonic function $u$ satisfies    $\nabla u \in L^p_{loc}(\Om)$ for every $p\in [2, p_{K^{min}})$. 
ii) There exist $\sigma\in \ms(\lambda,\Om)$ with 
$\sigma\in\{\sigma_1,\sigma_2\}$ and a  
$\sigma$-harmonic function with affine boundary conditions such that,  for every ball $B\subset\Om$
\begin{equation}\label{esplode}
\int_B |\nabla u| ^{p_{K^{min}}} dx = \infty \,.
\end{equation}
\end{theorem}

A key step to prove Theorem \ref{thm:sigma-thm0} is to prove the optimality of 
Astala's Theorem for a new class of  symmetric conductivities, 
specifically, for matrices of the form
\begin{equation}\label{matsim}
\sigma=\chi_{E_1} \diag(S_1,\lambda^{-1})+
\chi_{E_2} \diag(S_2,\lambda)
\,, \quad \text{ with } \lambda\le S_1, \, S_2 \le \lambda^{-1}\,,
\end{equation}
thus generalizing the isotropic case $S_1= \lambda^{-1}$, $S_2 = \lambda$, considered  in  \cite{AFS} and \cite{A}.

As a corollary of Theorem \ref{thm:sigma-thm0}, we prove that the bound 
\eqref{bestbound1} for non symmetric matrices too is optimal. 
Indeed, there exists  
$\sigma\in \ms(\lambda,\Om)$ of the form
\begin{equation}\label{matS}
\sigma=\chi_{E_1} \left(
\begin{array}{cc}
a & b \\
-b & a 
\end{array}
\right)
+
\chi_{E_2} \left(
\begin{array}{cc}
a & -b \\
b & a 
\end{array}
\right)
\,, \quad \text{ with } a=\lambda,\, b=\pm\sqrt{1-\lambda^2}\,
\end{equation}
and a $\sigma$-harmonic function $u$ such that 
the bound $\nabla u\in L^{p_{K_\lambda}}_{weak}$ is sharp 
(see Theorem \ref{mainthmS}).


Finally, a natural question, both in the symmetric and in the non symmetric case,  is wether 
there are other  two-phase critical coefficients, that is to say, 
two-phase coefficients $\sigma$ for which the bounds in 
Proposition \ref{ALN-bis} are attained and optimal in the sense of  \eqref{esplode}.  
In Theorem \ref{rigiditythm} we give a complete answer to this question, characterizing all the critical conductivities with fixed ellipticity.
In the symmetric case,  the critical conductivities
are given (up to rotations)  exactly by those  in \eqref{matsim} (for suitable partitions $E_1,\, E_2$).
In the non symmetric case,   the only critical conductivities are as  in \eqref{matS}. 
 
We remark that one can find optimal microgeometries  for $\sigma$'s which are not two-phase. The simplest example is given by a ``polycrystal'' like in the first example given in Leonetti and Nesi \cite{LN}. In that case $\sigma$ is symmetric, the eigenvalues are $\lambda$ and $\lambda^{-1}$ but the eigenvectors change from point to point.

\section{More about $\sigma$-harmonic functions and the Beltrami system}

\no
In the present section we review some well-known connections between $\sigma$-harmonic functions  
and the Beltrami system   which we use in the rest of the paper. We refer the interested reader to \cite{AN} for a more detailed presentation of the argument. 

\par
\subsection{Complex vs real formulation of a Beltrami system}

\no
Consider the Beltrami equation \eqref{beltrami}. It  can be rewritten in the equivalent form
\begin{equation}\label{beltramiGH}
Df^t H Df = G \det Df\,,
\end{equation}
where $G$ and $H$ are real matrix fields depending on $\mu$ and $\nu$.
Specifically,
\begin{align}\label{GH}
G=\frac{1}{d}
\left(
\begin{array}{cc}
|1+ \mu |^2-|\nu|^2&2\Im (\mu)\\
2\Im (\mu)&|1-\mu |^2-|\nu|^2
\end{array}
\right),
\\
\nonumber
H=\frac{1}{d}
\left(
\begin{array}{cc}
|1-\nu |^2-|\mu|^2&-2\Im (\nu)\\
-2\Im (\nu)&|1+\nu |^2-|\mu|^2
\end{array}
\right),
\end{align}
where
$$
d  =
\sqrt{(1-(|\nu|-|\mu|)^2)(1-(|\nu|+|\mu|)^2)}.
$$
We will refer to \eqref{beltrami} as well as to \eqref{beltramiGH} as the Beltrami system.
Let
$SL(2)$ be the subset  of $\M$ of the invertible matrices with determinant one, and let  $\Msdet = \Msim\cap SL(2)$. 
Notice that $G$ and $H$  belong to $\Msdet$   and they  are positive definite. In fact injective solutions to \eqref{beltramiGH} have a very neat geometrical interpretation. They are mapping $f:\Om\to\Om^{\prime}$ which are conformal, i.e., they preserves angles, provided  one uses the right scalar products, namely the one induced by $G$ in $\Om$ and $H$ in $\Om^{\prime}$. This interpretation has many consequences. We will get back to this point later in the paper.
Inversion of the above formulas yields
\begin{equation*}
\label{munu(HG)}
\mu =\frac{G_{11}-G_{22}+2 i G_{12}}{G_{11}+G_{22}+H_{11}+H_{22}},\quad \nu=\frac{H_{22}-H_{11}-2 i H_{12}}{G_{11}+G_{22}+H_{11}+H_{22}}.
\end{equation*}
By combining \eqref{GH} and \eqref{mu-nu(sigma)} we obtain a formula for $G$ and $H$ as 
functions of $\sigma$,
\begin{equation}\label{ghsigma}
G(\sigma)= \frac{1}{\sqrt{\det{\sigma^S}}}
\left(
\begin{array}{cc}
\sigma_{22}&-\frac{\sigma_{12}+\sigma_{21}}{2}\\
-\frac{\sigma_{12}+\sigma_{21}}{2}&\sigma_{11}
\end{array}
\right), \,
H(\sigma)=\frac{1}{\sqrt{\det{\sigma^S}}} 
\left(
\begin{array}{cc}
 \det \sigma&\frac{-\sigma_{12}+\sigma_{21}}{2}\\
\frac{-\sigma_{12}+\sigma_{21}}{2}&1
\end{array}
\right)\,,
\end{equation}
where $\dsp\sigma^S = \frac{\sigma + \sigma^T}{2}$.
Inversion of \eqref{ghsigma} gives 
\begin{equation}\label{sigma(G,H)bis}
\sigma=
\frac{1}{H_{22}}
\left(G^{-1}+ H_{12}J\right)
\end{equation}
where 
\begin{equation}\label{simp}
J= \left(\begin{array}{cc}
0&-1\\
1&0
\end{array}\right).
\end{equation}
Moreover, we can express $\sigma$ as a function of $\mu, \, \nu$ 
inverting the  algebraic system \eqref{mu-nu(sigma)},  
\begin{equation}\label{sigma(mu-nu)}
\sigma =
\left(
\begin{array}{ll}
\frac{|1-\mu|^2-|\nu|^2}{|1+\nu|^2-|\mu|^2} &  \frac{2\Im(\nu-\mu)}{|1+\nu|^2-|\mu|^2}\\
\, & \, \\
\frac{-2\Im(\nu+\mu)}{|1+\nu|^2-|\mu|^2} & \frac{|1+\mu|^2-|\nu|^2}{|1+\nu|^2-|\mu|^2}
\end{array}
\right)\,.
\end{equation}
Let us clarify the relationship between the Beltrami equation and $\sigma$-harmonic maps.
Given positive definite matrices $G$ and $H$ in $L^\infty(\Om;\Msdet)$, 
let $f=(u,v)$ be solution to \eqref{beltramiGH}. 
Then, the function $u$ is $\sigma$-harmonic, with $\sigma$  defined by \eqref{sigma(G,H)bis}. 
Conversely, given $\sigma$ satisfying the ellipticity conditions 
\eqref{ellipticity-sigma1}-\eqref{ellipticity-sigma2} and given 
a $\sigma$-harmonic function $u$, the map $f:= (u,v)$ solves \eqref{beltramiGH}, 
where $G$ and $H$ are defined by \eqref{ghsigma},  $v$ is such that 
\begin{equation}\label{stream}
J^T \nabla  v =  \sigma \nabla u, 
\end{equation} 
and  $J^T$ is the transpose of $J$ defined in \eqref{simp}.
The function $v$ is called {\it stream function} of $u$, and is defined up to additive constants.  
Moreover, $\dsp \|\nabla f\|_{L^p}$ is finite if and only if $\dsp \|\nabla u\|_{L^p}$ is finite.

\subsection{Different formulations of ellipticity and higher gradient integrability}
\label{sec:reformulation}
Here we introduce classical notions of ellipticity for elliptic and Beltrami equations, and we recall the fundamental summability results due to Astala \cite{A}
and some of its consequences due to Leonetti  and Nesi \cite{LN}.
From now on, we will always assume that the values of $\mu\,, \nu \,,G\,,H$ and $\sigma$ are related 
according to \eqref{mu-nu(sigma)} and \eqref{GH}.
\par
The ellipticity corresponding to any pair $\mu, \nu\in L^\infty(\Om;\C)$ satisfying \eqref{ellipticity-munu}
is the positive constant $k(\mu,\nu)$ defined by 
\begin{equation}\label{kBmn}
k(\mu,\nu) :=  \| |\mu|+|\nu| \|_{L^\infty}.
\end{equation}
An alternative  measure of ellipticity, that will be most convenient in our analysis, is provided by the following quantity  
\begin{equation}\label{KBmn}
K(\mu,\nu):= \frac{1+ k(\mu,\nu)}{1-k(\mu,\nu)}\,.
\end{equation}
Having in mind \eqref{GH}, we define $k(G,H)$ and $K(G,H)$ in the obvious way, i.e., 
\begin{equation}\label{defestesa}
k(G,H) = k(\mu,\nu), \quad K(G,H) = K(\mu,\nu)\,,
\end{equation}
and whenever no confusion may arise, we will omit the dependance on their argument.
In the next proposition we give a  more explicit formula for such ellipticity. 
We will denote by  $g(x)$ e $h(x)$ the maximum eigenvalue  of $G(x)$ and $H(x)$, respectively.  

\begin{proposition}\label{prodotto-autovalori}
Let $G,\, H\in L^\infty(\Om;\Msdet)$ be positive definite. Then 
\begin{equation}\label{formula-K}
K= \| g \, h\|_{L^\infty(\Om)}\,. 
\end{equation}
\end{proposition}

\begin{proof}
A direct computation shows that the maximum eigenvalues of $G$ and $H$ are given by
$$
g=\frac{\sqrt{(1-|\nu|+|\mu|)(1+|\nu|+|\mu|)}}{\sqrt{(1+|\nu|-|\mu|
))(1-|\nu|-|\mu|))}},
\qquad
h=\frac{\sqrt{(1+|\nu|-|\mu|)(1+|\nu|+|\mu|)}}{\sqrt{(1-|\nu|+|\mu|
))(1-|\nu|-|\mu|))}}.
$$
Therefore
$\dsp
gh= \frac{1+|\mu|+|\nu|}{1-(|\mu|+|\nu|)},
$
which yields  
$$
\| g h\|_{L^\infty}= \frac{1 + \| |\mu|+|\nu| \|_\infty}{1 -  \| |\mu|+|\nu| \|_\infty} = \frac{1 + k}{1 - k} = K\,.
$$
\end{proof}
Next, we relate the ellipticity bounds for the second 
order elliptic operator \eqref{elliptic} with the ellipticity of the associated Beltrami equation. 
Following the notation of \eqref{defestesa}, we set 
$K(\sigma):= K(G,H)$, where $G,H$ and $\sigma$ are related by \eqref{ghsigma}-\eqref{sigma(G,H)bis}.
The following result has been proved in \cite{LN} and \cite{AN}; for the reader's convenience, we give here a  proof based on Proposition~\ref{prodotto-autovalori}. 

\begin{proposition}\label{ALN}
Let $\lambda\in(0,1]$. 
For each
$\dsp\sigma\in \ms\left(\lambda,\Om\right)$ we have
\begin{equation}\label{bestbound}
 K(\sigma)\leq \frac{1+\sqrt{1-\lambda^2}}{\lambda}\,.
\end{equation}
If in addition $\sigma$ is symmetric, then 
\begin{equation}\label{bestbound2}
K(\sigma)\leq \frac{1}{\lambda} \,.
\end{equation}
\end{proposition}

\begin{proof}
Let $\lambda_1,\lambda_2$ be the eigenvalues of $\sigma^S$, with $\lambda_1\leq\lambda_2$. 
Then, from the assumption $\dsp \sigma\in \ms\big(\lambda,\Om\big)$ and the 
relationship 
$$
(\sigma^{-1})^S=\frac{\det\sigma^S}{\det\sigma}(\sigma^S)^{-1}\,,
$$
it follows 
\begin{align}
\label{disug1}\lambda_2\geq &\lambda_1\geq\lambda \,,\\
\label{disug2}\frac{\det\sigma^S}{\lambda_2\det\sigma}= &\frac{\lambda_1}{\det\sigma} \geq\lambda\,.
\end{align}
Next let $g$ and $h$ be the largest eigenvalue of $G$ and $H$ respectively. 
By \eqref{sigma(G,H)bis}, it is readily seen that 
$$
\sigma^S=\frac{1}{H_{22}}G^{-1} \,.
$$
and hence
\begin{equation}\label{formula-g}
g=\frac{1}{H_{22}}\frac{1}{\lambda_1}=\frac{\sqrt{\det\sigma^S}}{\lambda_1}\,.
\end{equation}
From \eqref{ghsigma} it follows
\begin{equation}\label{eq-h}
h+\frac{1}{h}=\frac{1}{\sqrt{\det\sigma^S}}(\det\sigma +1)\,.
\end{equation}
Set 
$
\displaystyle P:=\frac{\det\sigma +1}{\sqrt{\det\sigma^S}}\,.
$ 
Solving \eqref{eq-h} and choosing the root which is bigger than one, yields
\begin{equation}\label{formula-h}
h=\frac{P+\sqrt{P^2-4}}{2}\,.
\end{equation}
Then, using \eqref{formula-g}-\eqref{formula-h} and the inequalities \eqref{disug1}-\eqref{disug2}, 
we obtain the following upper bound for $gh$

\begin{align*}
gh & = \frac{1}{2\lambda_1}\left[\det\sigma +1+\sqrt{(\det\sigma+1)^2-4\det\sigma^S}\,\right] \\
     &\leq \frac{1}{2\lambda_1}
        \left[\frac{\lambda_1}{\lambda}+1+\sqrt{\Big(\frac{\lambda_1}{\lambda}+1\Big)^2-4\lambda_1^2}\,\right] \\
     &=\frac{1}{2}\left(\frac{1}{\lambda}+\frac{1}{\lambda_1}+\sqrt{\frac{1}{\lambda^2}+\frac{1}{\lambda_1^2}+
          \frac{2}{\lambda\lambda_1}-4} \,\right) \\
     & \leq \frac{1}{2}\left(\frac{2}{\lambda}+\sqrt{\frac{4}{\lambda^2}-4}\,\right) \\
     & = \frac{1+\sqrt{1-\lambda^2}}{\lambda}\,.
\end{align*}
\par

Now suppose that $\sigma$ is symmetric and denote by $\lambda_1$ and $\lambda_2$ its eigenvalues, 
with $\lambda_1\leq\lambda_2$. 
Since $\sigma\in \ms\big(\lambda,\frac{1}{\lambda},\Om\big)$, we have
\begin{equation}\label{stime-autovalori}
\lambda\leq\lambda_1\leq\lambda_2\leq\frac{1}{\lambda}\,.
\end{equation}
Formula \eqref{ghsigma} reduces itself to
$$
G=\sqrt{\det\sigma}\sigma^{-1}\,,\qquad 
H=\frac{1}{\sqrt{\det{\sigma}}} 
\left(
\begin{array}{cc}
 \det \sigma& 0\\
                0 &1
\end{array}
\right)\,.
$$
Therefore 
\begin{equation}\label{ellisim}
 g= \frac{1}{\lambda_1}\sqrt{\det\sigma}, \qquad 
 h=\frac{1}{\sqrt{\det\sigma}}\max\{\lambda_1 \lambda_2,1\}.
\end{equation}
In the case when $\lambda_1\lambda_2\leq 1$, we find 
$$
K=\Big\|\frac{1}{\lambda_1}\Big\|_{L^\infty}\leq \frac{1}{\lambda}\,.
$$
If otherwise $\lambda_1\lambda_2\geq 1$, we have
$$
K=\|\lambda_2\|_{L^\infty}\leq \frac{1}{\lambda}\,.
$$
\end{proof}
\no
In the next Proposition we look at conductivities $\sigma$ attaining the bounds \eqref{bestbound} and \eqref{bestbound2}.

\begin{proposition}\label{optimality}
Let $\dsp\sigma\in \ms\left(\lambda,\Om\right)$ 
for some $\lambda\in(0,1]$. Then the bound \eqref{bestbound} is attained
if and only if on a set of positive measure there holds
\begin{equation}\label{matriceS}
\sigma=\left(
\begin{array}{cc}
a & b \\
-b & a 
\end{array}
\right)\,, \quad \text{ with } a=\lambda,\, b=\pm\sqrt{1-\lambda^2}\,.
\end{equation}
Moreover, if $\sigma$ is symmetric \eqref{bestbound2} is attained if and only if either \eqref{ellipticity-sigma1} or 
\eqref{ellipticity-sigma2} is attained on a set of positive measure.
\end{proposition}

\begin{proof}
Keeping the notation introduced in the proof of Proposition \ref{ALN}, one can see that 
the bound \eqref{bestbound} is attained if and only if the inequalities \eqref{disug1}-\eqref{disug2} hold 
as equalities, namely, 
$$
\lambda_2 = \lambda_1= \lambda \,,\qquad
\frac{\lambda_1}{\det\sigma} =\lambda\,.
$$
It is readily seen that this is equivalent to \eqref{matriceS}. The symmetric case is left to the reader.
\end{proof}

We now recall the higher integrability results for gradients of solutions to 
\eqref{elliptic} and \eqref{beltrami}. 
For $K>1$, set $\dsp p_K:=\frac{2K}{K-1}$.
We start with the celebrated result in \cite{A}.

\begin{theorem}\label{astala}
Let $f\in H^1_{loc}(\Om;\C)$ be solution to \eqref{beltrami} with $K(\mu,\nu)>1$. Then
$$
\nabla f\in L^p_{loc}(\Om) \qquad \forall \,  p\in [2, p_{K(\mu,\nu)})\,.
$$  
\end{theorem} 
Recall that  $K_\lambda$ and $K_\lambda^{sym}$ are defined by \eqref{bestbound1}
and  \eqref{bestboundsym}, respectively. A straightforward computation yields 
\begin{equation}\label{pespli}
p_{K_\lambda} = {\frac{2+2 \sqrt{1-\lambda^2}}{1-\lambda +\sqrt{1-\lambda^2}}}, \qquad p_{K_\lambda^{sym}} = \frac{2}{1+\lambda}.
\end{equation}
As a consequence of Proposition \ref{ALN} and Theorem \ref{astala}, we obtain 
the following result which was proved in \cite{LN}, \cite{AN}.

\begin{theorem}\label{ALN2}
Let $\dsp\sigma\in \ms(\lambda,\Om)$ for some $\lambda\in(0,1]$.
Then, any solution $u\in H^1_{\rm loc}(\Om)$ to \eqref{elliptic} satisfies 
$$
\nabla u\in L^p_{loc}(\Om) \quad \forall\,  
p\in\left[2, p_{K_\lambda}\right), 
$$  
and, if  $\sigma\in L^\infty(\Om;\Msim)$, 
$$
\nabla u\in L^p_{loc}(\Om) \quad \forall\,  
p\in\left[2, p_{K^{sym}_\lambda}\right), 
$$
where $ p_{K_\lambda}$ and $p_{K_\lambda^{sym}}$ are given in \eqref{pespli}. 
\end{theorem} 

We are now ready to perform linear change of variables both in the domain and in the target space. It will be convenient to work with the real formulation of the equation \eqref{beltramiGH}. 
Let $A, B \in SL(2)$ and set
\begin{equation}\label{trasformatedet}
\tilde f(x):= A^{-1} f(Bx), \quad \tilde G(x) :=  B^t G(B x) B ,       \quad \tilde H(x) := A^{t} H(B x) A \,. 
\end{equation}
A  straightforward  computation shows that,
whenever $f:\Om\mapsto \R^2$ is solution to  \eqref{beltramiGH},  $\tilde f$   solves  
\begin{equation}\label{beltramiGHt}
D \tilde f^t  \tilde H D \tilde f = \tilde G \det D \tilde f.
\end{equation}
Clearly $\tilde f$ enjoys the same integrability properties as $f$.
This motivates the following definition,
\begin{equation}\label{teta}
K^{min}(G,H) := \min_{ A,B\in SL(2)} \| g(A,B)h(A,B)\|_{L^\infty}\,,
\end{equation}
where $g(A,B)$ and $h(A,B)$ denote the maximum eigenvalue of $\tilde G$ and $\tilde H$, respectively. 
Remark that $g(A) \ge c \|A\|^2$, $h(B)\ge c \|B\|^2$. Therefore, the minimum in \eqref{teta} is attained. 
Recalling \eqref{formula-K}, a straightforward generalization of Theorem \ref{astala} 
leads to the following result.

\begin{proposition}\label{kmin}
Let  $G,H\in \Msdet$ and let  $K^{min}(G,H)$ be defined as in \eqref{teta}.
Then any $f\in H^1_{\rm loc}(\Om;\R^2)$ solution to \eqref{beltramiGH} satisfies 
$$
\nabla f\in L^p_{loc}(\Om) \qquad \forall \,  p\in [2, p_{K^{min}(G,H)})\,.
$$   
\end{proposition}

\begin{remark}
From the point of view of $\sigma$-harmonic maps, Proposition \ref{kmin} may be 
rephrased by saying that any solution $u\in H^1_{\rm loc}(\Om)$ to \eqref{elliptic}, satisfies
$$
\nabla u\in L^p_{loc}(\Om) \qquad \forall \,  p\in [2, p_{K^{min}(\sigma)})\,,
$$  
where $K^{min}(\sigma)$ is defined in the obvious way, i.e., $K^{min}(\sigma):=K^{min}(G,H)$, 
and $G,H$ and $\sigma$ are related by \eqref{ghsigma}.
\end{remark}


\section{Examples of weak solutions with critical integrability properties}
 
 In \cite{F}, \cite{AFS}, the authors exhibit an example of weak solution to 
 \eqref{elliptic} with critical integrability properties. In their construction the 
 essential range of $\sigma$ consists of only two isotropic matrices, namely, 
 $ \sigma:\Om\mapsto \{K^{-1} I, K I\}$ with  $K >1$.
 In this section  we 
 generalize their  construction to the case 
\begin{equation}\label{eq:critical-matrices}
\sigma_1:= \diag(K, S_1), 
\qquad \sigma_2:= \diag (K^{-1}, S_2 ),
\qquad \qquad \frac{1}{K}\le S_i \le K,
\end{equation}
thus proving optimality of Astala's theorem for the whole class of matrices above. In Section \ref{tpc} we will show that such class cannot be further 
enlarged. 

We will need the following definition.

\begin{definition}{\rm
The family of {\it laminates of finite order} is the smallest family of probability 
measures $\caL(\M)$ on $\M$ such that 
\begin{itemize}
\item[(i)] $ \caL(\M)$ contains all Dirac masses; 
\item[(ii)]   if  $\sum_{i=1}^n \alpha_i \delta_{A_i}\in\caL(\M)$ 
 and 
$A_1= \alpha B + (1-\alpha)C \text{ with } {\rm rank}(B-C)=1$, 
 then the  probability measure 
$\sum_{i=2}^n \alpha_i \delta_{A_i} + \alpha_1\big(\alpha B + (1-\alpha)C\big)$
 is also contained in $ \caL(\M)$.
\end{itemize}

Given $\nu\in \caL(\M)$, we define the {\rm barycenter} $\bar\nu$ of $\nu$ as
$$
\bar\nu:= \int_{\M}M \, d\nu(M) \,.
$$
}
\end{definition}

\begin{theorem}\label{thm:afsz-generalized}
Let $\sigma_1,\sigma_2$ be as in \eqref{eq:critical-matrices}.
There exists a measurable matrix field 
$\sigma: \Om \to \{\sigma_1,\sigma_2\}$ such that the solution 
$u\in H^{1}(\Om)$ to 
\begin{equation}\label{final-goal}
\begin{cases}
\div  (\sigma \nabla u) = 0 & \text{ in } \Om \\
u(x) = x_1  & \text{ on } \partial \Om \\
\end{cases}
\end{equation}
satisfies for every ball $B\subset\Om$
$$
\int_B |\nabla u| ^{p_{K}} dx = \infty \,.
$$
\end{theorem}

The proof follows the strategy in \cite[Theorem 3.13]{AFS}, where 
the result is proved for $\sigma_1 = K I$, $\sigma_2 = K^{-1}  I$.
Here the main difference is that we work with coefficients that are not isotropic.
For the reader's convenience we shortly reproduce the arguments of 
\cite{AFS} 
pointing out the essential modifications.

{\bf Step 1} ({\it Reformulation of \eqref{final-goal} as a differential inclusion}). Recall that $u$ is solution to \eqref{final-goal} if and 
only if $u=f_1$ where $f=(f_1,f_2)$ is solution to the associated 
Beltrami equation. It is easily checked that, for $\sigma$ of the form 
\eqref{eq:critical-matrices}, the latter condition 
is equivalent to 

\begin{equation}\label{eq:inclusion}
Df \in E := E_1 \cup E_2 \,,
\end{equation}
where
$$
E_1 = \left\{
\begin{pmatrix}
T \\
J\sigma_1 T
\end{pmatrix}\,,
T \in\R^2
\right\}\,, \quad 
E_2 = \left\{
\begin{pmatrix}
T \\
J\sigma_2 T
\end{pmatrix}\,,
T \in\R^2
\right\}.
$$
The goal is to find a solution $f\in H^1(\Om;\R^2)$ to the differential inclusion \eqref{eq:inclusion} satisfying in addition the boundary condition 
$f_1(x)= x_1 $ on $\partial\Om$.

Next we define a setting where to apply the Baire category method. 
Fix $\delta > 0$  such that 
\begin{equation}\label{defdelta}
\delta < \left(\frac{ (1-1/K)(K-1)}{4  \max\{S_1,S_2\} K^2 }\right)^{\frac 12}, 
\end{equation}
and let
\begin{equation}\label{nuovocono}
\tilde E:=
E \cap 
 \left\{
\begin{pmatrix}
a_{11} & a_{12}  \\
a_{21} & a_{22}
\end{pmatrix}\in \M\,:
|a_{12}| < \delta a_{11}
\right\} .
\end{equation}
Notice that the introduction  of the small parameter $\delta$  enforces the solutions to have gradient pointing in a direction  relatively close to 
$I$. 
This property hides  the anisotropy of the coefficients $\sigma_i$, and allows us to follow the strategy of \cite{AFS}.  
\no
Define $\U$ as the interior of the quasiconvex hull of $\tilde E$
(defined as the set of 
range of weak limits in $L^{2}$ of solutions to 
\eqref{eq:inclusion}). 
The following characterization of $\overline \U$ holds
$$
\tilde E^{lc,1} = \overline \U = \tilde E^{pc}\,,
$$
where $\tilde E^{lc,1}$ and  $\tilde E^{pc}$ denote the first lamination hull and the 
polyconvex hull of $\tilde E$, respectively. We refer to \cite[Lemma 3.5]{AFS} for the proof of the identity above and for the notion of
first lamination hull and polyconvex hull. 

Set 
$$
X_0 = \{
f \in W^{1,\infty}(\overline\Om;\R^2) : f \text{ piecewise affine}, 
Df\in\U \text{ a.e.}, \  f|_{\partial\Om}= x  
\},
$$
 let $X$ be its closure in the weak topology of $H^1$,
and denote by $(X,w)$ the set $X$ endowed with the weak topology $w$ 
of $H^1$. 
Remark that $I\in \U$ and therefore 
the set $X$ is not empty as it contains the map $f(x) = x$.

%


{\bf Step 2} ({\it Existence of solutions by the Baire category method}).  
The existence of  solutions to the differential inclusion is proved by an application of the Baire category method, and is based on the fact that the gradient operator  $D: X \mapsto  L^2(\Om;\M)$  is a Baire-1 mapping, i.e., the 
pointwise limit of continuous mappings. 
We refer to \cite[page 57]{kirc} and references therein for further clarifications 
on this subject. The existence result is stated in the next theorem. We refer to 
\cite[Lemma 3.7]{AFS} for its proof.

\begin{theorem}\label{norme-equivalenti}
The space $(X,w)$ is compact and metrizable. Each $f\in X$ satisfies $f\in \overline \U$ 
and  $f|_{\partial\Om}=x$. 
The metric $d$ on $X$ is equivalent to the metrics induced by the 
$L^2$ and $L^\infty$ norms. 
Moreover, the points of continuity of the map 
$\dsp  D: (X,w)\to L^2(\Om ;\M)$ form a residual set in $(X,w)$. Finally,
any point of continuity $f\in X$ of $D$  satisfies 
$Df \in E_1 \cup E_2$. 
\end{theorem}
\no
We deduce that the set of solutions to the differential inclusion 
\eqref{eq:inclusion} is residual in $(X,w)$.
The proof of Theorem \ref{thm:afsz-generalized} is then a consequence of Theorem  \ref{norme-equivalenti} and of the 
following theorem. 

\begin{theorem}\label{thm-residual}
The set
$$\dsp
\Big\{f\in X \, : \, \int_B |Df|^{p_K} = +\infty 
\text{ for all balls }B\subset \Om   
\Big\}
$$
is residual in $X$. 
\end{theorem}
\no
Theorem \ref{thm-residual} is proved following the same strategy of the proof of  Corollary 3.12 in \cite{AFS}. We recall that in \cite{AFS} the isotropic case $S_1 = K$, $S_2= 1/K$ is considered. In the present setting the proof is identical except for the proof of a key ingredient (namely, \cite[Proposition 3.10]{AFS}). Therefore, we only state and proof such  result in Lemma 
\ref{lemma:barycenter} below.  
 For this purpose it is convenient to  introduce some notation.
Given a matrix $A= (a_{ij}) \in\M$, we denote by $A_d$ and 
$A_a$ its diagonal and anti-diagonal part, namely
\begin{equation}\label{diag-antidiag}
A_d :=
\left(
\begin{array}{cc}
a_{11} & 0 \\
0  &   a_{22}
\end{array}
\right)  \,, 
\quad
A_a :=
\left(
\begin{array}{ll}
0 & a_{12} \\
a_{21} & 0
\end{array}
\right) \,.
\end{equation}
Moreover we will identify $A_d$ and $A_a$ with points 
of $\R^2$: 
$A_d = (a_{11}, a_{22})$, 
$A_a =(a_{12},a_{21})$.
Finally, we denote by $\cs$ the following cone of $\R^2$. 
\begin{equation}
\label{cono1}
\cs= \{ t(a,a/K) + (1-t)(a,Ka): t\in(0,1), a\in\R^+ \}.
\end{equation}

\begin{lemma}\label{lemma:barycenter}
Every $A\in \U$ is the barycenter of a sequence of laminates 
of finite order
$\nu_n\in \caL$ 
such that {\rm supp }$\nu_n\subset\U$ and 
\begin{equation}\label{staircase-laminates}
\lim_{n\to \infty} \int_{\M}|M|^{p_K} d\nu_n(M) = \infty \,.
\end{equation}
\end{lemma}

\begin{proof}

The proof of Lemma \ref{lemma:barycenter} follows the strategy of the proof of
\cite[Proposition 3.10]{AFS}, where the particular case of $S_1 = K$ 
and $S_2=1/K$ is considered.  
In \cite{AFS} it is first showed that the identity matrix is the barycenter of 
a sequence of laminates of finite order 
satisfying  \eqref{staircase-laminates} and with 
support on $\partial \cs$, 
where $\cs$ is the cone defined by \eqref{cono1}. 
The proof is based on the construction of the so-called staircase laminates, which 
was originally made in \cite{F}.
Then, they extend the result to all other matrices by using the conformal invariance of the quasiconvex hull. In our case $\U$ does not enjoy conformal invariance, due to the anisotropy of the coefficients $\sigma_i$. 
Therefore, we have to proceed in a different way.

By slightly modifying the staircase construction in \cite{F}, \cite{AFS} 
(in fact only a finite number of steps at the beginning of the staircase) one can easily show that each point in 
$\cs$ can be obtained as the barycenter of 
a sequence of laminates of finite order, 
satisfying  \eqref{staircase-laminates} and with 
support on $\partial \cs$. Moreover, by a suitable shift of the support, one 
can obtain that these measures have support in the interior of the cone $\cs$.

Now let $A=(a_{ij})\in \U$. 
We claim that $A$ is rank-one connected to a diagonal matrix $Q=(q_{ij})= Q_d \in \cs$ and  we conclude the proof. 
Arguing as in \cite[Remark 3.6]{AFS} it is easy to show that $Q\in \U$ (that is to say, $Q$ belongs to the interior of the quasiconvex hull), and that
 $A$ belongs to a suitable segment $[P,Q]$ still contained in $\U$, i.e.,  
$A=\tau P + (1-\tau) Q$ for some $\tau\in (0,1)$.
Since $Q\in \cs$, $Q$ is the barycenter of a sequence of laminates 
$\nu_n = \sum \lambda_j \delta_{A_j}$ 
supported in $\U$ and satisfying  \eqref{staircase-laminates}.
The required laminates can then be defined as 
$$
\tilde\nu_n=
\tau \delta_P + (1-\tau) \sum \lambda_j \delta_{A_j } \,.
$$ 
 We conclude by proving the claim. 
From \eqref{nuovocono} it follows that  $A_d\in \cs$. 
The condition of rank-one connectedness reads as 
\begin{equation}
(a_{11}-q_{11})(a_{22}-q_{22}) = a_{21}a_{12} .
\end{equation}

\begin{figure}
\centering
\psfrag{c1}[B1][B1][0.9][0]{$\cs$}
\psfrag{c2}[B1][B1][0.9][0]{$\cs_2$}
\psfrag{ad}[B1][B1][0.9][0]{$A_d$}
\psfrag{aa}[B1][B1][0.9][0]{$A_a$}
\psfrag{qd}[B1][B1][0.9][0]{$Q_d$}
\psfrag{qa}[B1][B1][0.9][0]{$Q_a$}
\psfrag{q11}[B1][B1][0.9][0]{$q_{11}$}
\psfrag{q22}[B1][B1][0.9][0]{$q_{22}$}
\psfrag{a11}[B1][B1][0.9][0]{$a_{11}$}
\psfrag{a22}[B1][B1][0.9][0]{$a_{22}$}
\psfrag{a12}[B1][B1][0.9][0]{$a_{12}$}
\psfrag{a21}[B1][B1][0.9][0]{$a_{21}$}
\psfrag{k1}[B1][B1][0.9][0]{$(a,a/K)$}
\psfrag{k2}[B1][B1][0.9][0]{$(a,Ka)$}
\psfrag{s1}[B1][B1][0.9][0]{$(a,-S_1 a)$}
\psfrag{s2}[B1][B1][0.9][0]{$(a,-S_2 a)$}
\includegraphics[scale=0.9]{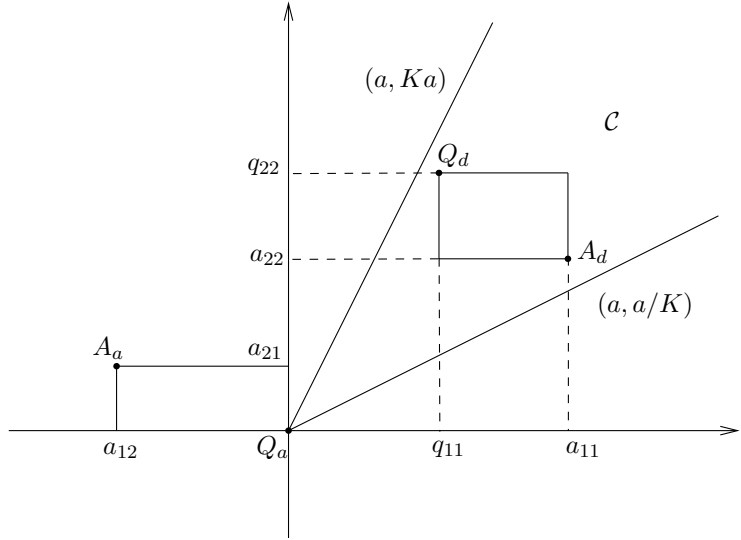}
\caption{The rank-one connected points $A$ and $Q$.}
\label{due-coni}
\end{figure}

\no
This is equivalent to the fact that  the two rectangles 
with sides parallel to the axis and 
diagonal
$Q_d A_d$ and $Q_a A_a$  
have the same signed area (see Figure \ref{due-coni}). 
Notice that the sign of the areas is given by the sign of the slope of 
$Q_d A_d$ and of $Q_a A_a$, respectively.  
Define $h(A_d,Q_d)$ as the signed area of the corresponding rectangle and remark 
that it is a continuous function. Given $A$, the problem is to find $Q_d$ 
such that 
\begin{equation}\label{harea}
h(A_d,Q_d)=a_{21}a_{12}.
\end{equation}
Notice that 
$$
\{ h(A_d,Q_d), Q_d \in \overline\cs\} = [m,\infty) 
$$
for a suitable negative $m<0$ depending on $A_d$.
Therefore, If $a_{21}a_{12}>0$ we can  always solve \eqref{harea}. 
Assume instead that $a_{21}a_{12}<0$ like in Figure \ref{due-coni}. 
Let $\tilde h(A_d)$ be the infimum of $h$ over $Q_d$. 
For a fixed $a_{11}$,  it is easy to see that $\tilde h$ attains its maximum for 
$a_{22} = a_{11}$. In this case, the optimal $Q_d$ is given by
$$
Q_d = \frac 12 \left(  a_{11} + \frac{a_{11}}{K} ,    K a_{11} + a_{11} \right),
$$
and 
$$
\max_{a_{22}} \tilde h(A_d) = - \frac{a_{11}^2}{4}(1-1/K)(K-1).
$$
Therefore \eqref{harea} has a solution whenever 
\begin{equation}\label{finade}
- \frac{a_{11}^2}{4}(1-1/K)(K-1) < a_{21}a_{12}.
\end{equation}
From \eqref{nuovocono} it follows that
$$
|a_{21}| < \max\{S_1,S_2\} K  \delta  a_{22} ,
$$
and hence
$$
|a_{12}a_{21}| < \max\{S_1,S_2\} K  \delta  a_{22} |a_{12}|<
\max\{S_1,S_2\} K  \delta^2  a_{11}a_{22} <
\max\{S_1,S_2\} K  \delta^2 K a_{11}^2 .
$$
By the very definition  \eqref{defdelta} of $\delta$ we deduce that 
$$
\frac{a_{11}^2}{4}(1-1/K)(K-1)> \max\{S_1,S_2\} K^2  \delta^2  a_{11}^2,
$$
so that \eqref{finade} holds, and the proof is completed.

\end{proof} 
 

\section{Two phase Beltrami coefficients}

\no
In the present section we focus on  two-phase Beltrami coefficients. In this class, we find the ellipticity $K^{min}$ defined in \eqref{teta} and we characterize the  Beltrami coefficients for which $K=K^{min}$. 
From now on, to easy notation, we will omit the dependence on $G$ and $H$ in the ellipticity constants.


\subsection{Two-phase Beltrami equation}
Let $E_1$ be a measurable subset of $\Om$ and let $E_2:= \Om \setminus E_1$. 
Fix  $\{G_1,\, G_2, \, H_1,\, H_2 \}\subset\Msdet$ positive definite (symmetric and with determinant one), and consider the functions 
\begin{equation}\label{2fasi}
G:= \chi_{_{E_1}} G_1+ \chi_{_{E_2}} G_2 , \qquad H:= \chi_{_{E_1}}H_1  + \chi_{_{E_2}}  H_2,
\end{equation}
where $\chi_{_{E_1}}$ and $\chi_{_{E_2}}$ are the characteristic functions of $E_1$ 
and $E_2$, respectively.
From \eqref{formula-K} it follows that for $G$ and $H$ of the form \eqref{2fasi}, one has
$$
K = \max\{|g h| \res E_1, \, |g h| \res E_2\}= \max\{g_1 h_1\, , \,  g_2 h_2\}.
$$
where $g_i$ and $h_i$ denote the largest eigenvalue in  $E_i$ of $G$ and $H$, respectively. 
Set
\begin{equation*}
\hat{K}:=\sqrt{g_1h_1g_2h_2} \,.
\end{equation*}

\begin{lemma}\label{primadisug}
The following inequality holds
\begin{equation*}
K^{min}\leq \hat{K} \le K \,.
\end{equation*}
\end{lemma}

\begin{proof}
The  inequality $\hat{K} \le K$ is trivial. 
Let us prove that $K^{min}\leq \hat{K}$. 
Without loss of generality we may assume that $g_1h_1\geq g_2h_2$. 
Set 
$$
\dsp\lambda:=\sqrt{\frac{g_2h_2}{g_1h_1}}\le 1.
$$ 
We can have either of the following cases: $h_1<\max\{g_1,g_2,h_1,h_2\}$ or $h_1=\max\{g_1,g_2,h_1,h_2\}$.
Suppose we are in the first case.  
Up to a diagonalization, $G_1$ is of the form
$$
G_1=\left(
\begin{array}{ll}
g_1 & 0\\
0     & \frac{1}{g_1}
\end{array}
\right)\,.
$$
We want to use the change of variables \eqref{trasformatedet}, and we recall that $g(A,B)$ and $h(A,B)$ denote the maximum eigenvalue of $\tilde G$ and $\tilde H$, respectively.
We choose 
$$\dsp B=\left(
\begin{array}{ll}
\sqrt{\lambda} & 0\\
0     & \frac{1}{\sqrt{\lambda}}
\end{array}
\right),
$$
 and  $A=I$.
Then $\dsp g_1(A,B)=\lambda g_1$ and $\dsp g_2(A,B)\leq \frac{1}{\lambda}g_2$. 
Therefore 
$$
g_1(A,B)h_1(A,B) =\hat{K} \quad\text{and }\quad g_2(A,B)h_2(A,B) \leq\hat{K}.
$$  
We deduce
$$
K^{min}\leq g_1(A,B)h_1(A,B)=\hat{K}\,.
$$
Suppose now that $h_1=\max\{g_1,g_2,h_1,h_2\}$. 
Then, after diagonalization of $H_1$, we  choose $B=I$ and $A=\left(
\begin{array}{ll}
\lambda & 0\\
0     & \frac{1}{\lambda}
\end{array}
\right)$, and we proceed as before. 
\end{proof}

\begin{remark} \label{rigidity0}
{\rm
A direct  consequence of Lemma \ref{primadisug} is that 
$
K^{min} < K
$
whenever  $g_1h_1<g_2h_2$.
}
\end{remark}


\begin{proposition}\label{Kmin}
The following formula for $K^{min}$ holds:
\begin{equation}\label{formulakmin}
K^{min}(G,H)=\sqrt{g_2(G_1^{-1/2},H_1^{-1/2})h_2(G_1^{-1/2},H_1^{-1/2})}\,.
\end{equation}
\end{proposition}

\begin{proof}

In view of Lemma \ref{primadisug}, it is enough to prove  that for each 
$A,B\in SL(2)$ we have 
\begin{equation}\label{step}
g_2(G_1^{-1/2},H_1^{-1/2})h_2(G_1^{-1/2},H_1^{-1/2})
\leq g_1(A,B)h_1(A,B)g_2(A,B)h_2(A,B)\,.
\end{equation}
For this purpose, we show that if $G_1=H_1=Id$, then for each 
$A,B\in SL(2)$
\begin{align}
\label{g1g2min} &\dsp  g_2 \leq g_1(A,B)g_2(A,B)\,,\\
\label{h1h2min} &\dsp  h_2 \leq h_1(A,B)h_2(A,B)\,.
\end{align}
Let $B\in SL(2)$ and set 
$$
\tilde G_1:=B^T B \,,\quad \tilde G_2:=B^TG_2 B \,.
$$
For every $v\in\R^2$ we have 
$$
\frac{1}{g_1(A,B)}\|v\|^2 \leq \langle \tilde G_1 v,v\rangle = \|Bv\|^2,
$$
and hence
$$
g_2(A,B)=\sup_{\|v\|\le 1}\langle \tilde G_2 v,v\rangle=\sup_{\|v\|\le 1}\langle G_2 Bv,Bv\rangle
\geq \frac{g_2}{g_1(A,B)}\,,
$$
which proves \eqref{g1g2min}. The proof of \eqref{h1h2min} is fully analogous.


\end{proof}

\subsection{Gradient integrability and critical coefficients}

In the next proposition we will show that  if the bound $K^{min}\le K$ is achieved, then $G_i$ and $H_i$ can be simultaneously diagonalized.   

\begin{proposition}\label{simudiago}
Let $G$ and $H$ be as in \eqref{2fasi} and assume that $\dsp K^{min}=\hat{K}$. Then, there exist 
$A,B\in O(2)$ such that 
\begin{align}\label{simudiagoeq1}
& A^TG_1A:= \left(
\begin{array}{cc}
{g_1}& 0\\
0 & \frac{1}{g_1}
\end{array}
\right)\,,
\qquad
A^TG_2A:= \left(
\begin{array}{cc}
\frac{1}{g_2} & 0\\
0 & g_2
\end{array}
\right)\,,
\\
\label{simudiagoeq2}
& B^T H_1B:= \left(
\begin{array}{cc}
{h_1}& 0\\
0 & \frac{1}{h_1}
\end{array}
\right)\,,
\qquad
 B^T H_2B:= \left(
\begin{array}{cc}
\frac{1}{h_2} & 0\\
0 & h_2
\end{array}
\right).
\end{align}
\end{proposition}

\begin{proof}
We can always assume that $G_1$ and $H_1$ are as in 
\eqref{simudiagoeq1}-\eqref{simudiagoeq2}. We prove that, in this case,  also $G_2$ is  diagonal 
(For $H_1$ and $H_2$ we argue exactly in the same way).  
Set
\begin{align*}
& \hat{B}:= G_1^{-\frac 12}\,,
\qquad \hat G_1:= \hat{B} G_1\hat{B} =I\,, 
\qquad \hat G_2:= \hat{B} G_2 \hat{B}\,, \\
& \hat{A}:= H_1^{-\frac 12}\,, 
 \qquad \hat H_1 = \hat{A} H_1 \hat{A}=I\,, 
 \qquad \hat H_2 = \hat A H_2 \hat A \,. 
\end{align*}
Since  $\hat h_2 \le h_1 h_2$, $\hat g_2 \le g_1 g_2$ and  recalling Proposition \ref{Kmin} we have 
$$
(K^{min})^2 = \hat g_1 \hat h_1 \hat g_2 \hat h_2 =\hat g_2\hat h_2 \le \hat g_2 h_1 h_2 \le h_1 h_2 g_1 g_2 = \hat{K}^2,
$$ 
where $\hat g_i$ and $\hat h_i$  are the largest eigenvalues of $\hat G_i$ and $\hat H_i$. Since  $\dsp K^{min}=\hat{K}$, all the above inequalities are indeed equalities, and in particular 
$\hat g_2 = g_1 g_2$, that  implies  $G_2$ diagonal. 
\par
We are left to show that $e_2$ is the eigenvector associated with $g_2$. 
Arguing by contradiction, we assume that 
$$
G_2= \left(
\begin{array}{cc}
g_2 & 0\\
0 & \frac{1}{g_2}
\end{array}
\right).
$$
\no
Without loss of generality we may suppose that $g_1\le g_2$ and we set
$$
 \hat{B}:= G_1^{-\frac 12}\,,
\qquad \hat G_1:= \hat{B} G_1\hat{B} =I\,, 
\qquad \hat G_2:= \hat{B} G_2 \hat{B}. 
$$
It can be easily checked that  $\hat g_i < g_i$, that (recall  $K^{min}=\hat{K}$) provides  the following contradiction 
$$
(K^{min})^2 \le \hat g_1  h_1 \hat g_2  h_2 <  g_1 h_1 g_2  h_2  = \hat{K}^2.
$$ 
 \end{proof}
 
We are in a position to show that, for two phase coefficients, Proposition \ref{kmin} is sharp.

\begin{theorem}\label{mainthmGH}
Let $G_1,\, G_2, \, H_1,\, H_2 \in \Msdet$,  and let $K^{min}$ be as defined in \eqref{formulakmin}. Then we have
\begin{itemize}
\item[i)] Let $G$ and $H$ be as in \eqref{2fasi}. Every solution $f\in H^1_{\rm loc}(\Om;\C)$  to \eqref{beltrami} belongs to $L^p_{loc}(\Om;\C)$ for every $p\in [2, p_{K^{min}})$;
\item[ii)] There exist $G$ and $H$ as in \eqref{2fasi},  and a corresponding solution 
$f \in H^1_{loc}(\Om;\C)$ to  $\eqref{beltrami}$ with 
$\nabla f \notin L^{p_{K^{min}}}(B;\M)$ for every disk $B\subset \Om$.
\end{itemize}
\end{theorem}
\begin{proof}
The first part of the Theorem is a particular case of Proposition \ref{kmin}, so we pass to the proof of ii). By the definition of $K^{min}$ and by 
Proposition \ref{simudiago}, we can always assume that $G$ and $H$ are diagonal as in \eqref{simudiagoeq1}, \eqref{simudiagoeq2}, with $g_i h_i = K^{min}$. A straightforward computation shows that the corresponding $\sigma$, defined according to  \eqref{sigma(G,H)bis}, takes the form 
\begin{equation*}
\sigma_1:= \diag\left(S_1,K^{min}\right), \qquad 
\sigma_2:= \diag\left(S_2,\frac{1}{K^{min}}\right), 
\qquad \qquad  K^{-1}  \le S_i \le K.
\end{equation*}
 Therefore,  ii) follows from Theorem \ref{thm:afsz-generalized}.
 \end{proof}
 
 \begin{remark}\label{f1aff}
 {\rm
 In ii) of Theorem \ref{mainthmGH}, we can also enforce that $f_1$ satisfies  suitable  
 affine boundary conditions. 
 }
 \end{remark}
 
 \section{Two phase conductivities}\label{tpc}
In this part we study the gradient summability of $\sigma$-harmonic functions corresponding to two phase conductivities.
 Let $E_1$ be a measurable subset of $\Om$ and  let 
 $E_2:= \Om \setminus E_1$. We assume that both $E_1$ and $E_2$ have positive measure.
Given  positive definite matrices  $\sigma_1,\, \sigma_2 \in \Md$, define 
\begin{equation}\label{2fasisig}
\sigma:= \chi_{_{E_1}} \sigma_1+ \chi_{_{E_2}} \sigma_2.
\end{equation}
Set 
$$
K^{min}=K^{min}(\sigma):= K^{min}(G(\sigma), H(\sigma)).
$$

\subsection{Main results and optimality of the bound  \eqref{bestbound1}}
We can now rephrase Theorem  \ref{mainthmGH} (see also Remark \ref{f1aff}) in terms of the coefficient $\sigma$.
 
 \begin{theorem}\label{thm:sigma-thm}
Let $\sigma_1,\, \sigma_2 \in \Md$ be positive definite.
\begin{itemize}
\item[i)] Let $\sigma$ be a two phase conductivity as in \eqref{2fasisig}. Every solution $u\in H^1_{loc}(\Om)$  to \eqref{elliptic} satisfies   $\nabla u \in L^p_{loc}(\Om)$ for every $p\in [2, p_{K^{min}})$;
\item[ii)] There exist $\sigma$  as in \eqref{2fasisig} and $(v_1,v_2)\in\R^2$ such that the solution $u \in H^1_{loc}(\Om)$ to 
\begin{equation*}
\begin{cases}
\div  (\sigma \nabla u) = 0 & \text{ in } \Om \\
u(x) = v_1 x_1 + v_2 x_2 & \text{ on } \partial \Om \\
\end{cases}
\end{equation*}
satisfies for every ball $B\subset\Om$
$$
\int_B |\nabla u| ^{p_{K^{min}}} dx = \infty \,.
$$
\end{itemize}
\end{theorem}

We are in a position to  prove that the bound in \eqref{bestbound1} is achieved by a suitable conductivity $\sigma$ of the type 
\begin{equation}\label{sigmaS}
\sigma=\chi_{E_1} \left(
\begin{array}{cc}
a & b \\
-b & a 
\end{array}
\right)
+
\chi_{E_2} \left(
\begin{array}{cc}
a & -b \\
b & a 
\end{array}
\right)
\,, \quad \text{ with } a=\lambda,\, b=\pm\sqrt{1-\lambda^2}\,.
\end{equation}
\begin{theorem}\label{mainthmS}
There exist $\sigma$  as in \eqref{sigmaS},  and a corresponding solution 
$u \in H^1_{loc}(\Om)$ of  \eqref{elliptic} with affine boundary conditions such that 
$\nabla u \notin L^{p_{K_\lambda}}(B)$ for every disk $B\subset \Om$, where $p_{K_\lambda}$ is given by \eqref{pespli}.
\end{theorem}

\begin{proof}
By \eqref{ghsigma} we have $G_i(\sigma) = I$ for $i=1,2$, and 
$$
H_1 = \left(
\begin{array}{ll}
{\lambda}^{-1} & \sqrt{1-\lambda^2} \\
\sqrt{1-\lambda^2} & {\lambda}^{-1} 
\end{array}
\right),
\quad
H_2 = \left(
\begin{array}{ll}
{\lambda}^{-1} & -\sqrt{1-\lambda^2} \\
- \sqrt{1-\lambda^2} & {\lambda}^{-1} 
\end{array}
\right).
$$
Therefore $K(\sigma)=K^{min}(\sigma)=
\frac{1+\sqrt{1-\lambda^2}}{\lambda} = K_\lambda$. 
We conclude in view of  Theorem \ref{thm:sigma-thm}. 
\end{proof}

Finally, we fix the ellipticity $\lambda$ and we characterize the pairs $(\sigma_1,\sigma_2)$ corresponding to solutions with critical gradient integrability. 

\begin{theorem}\label{rigiditythm} 
Let $p_{K_\lambda}$ and $p_{K_\lambda^{sym}}$ be as in \eqref{pespli}. 
\begin{itemize}
\item[i)] Let $\sigma\in \ms\left(\lambda,\Om\right)$ be a two phase 
conductivity as in \eqref{2fasisig} such that  there exists a solution 
$u\in H^1(\Om)$ of \eqref{elliptic} with 
$\nabla u \notin 
L_{loc}^{p_{K_\lambda}}(\Om;\R^2)$; then  $\sigma$ takes the following form
\begin{equation*}
\sigma=\chi_{E_1} \left(
\begin{array}{cc}
a & b \\
-b & a 
\end{array}
\right)
+
\chi_{E_2} \left(
\begin{array}{cc}
a & -b \\
b & a 
\end{array}
\right)
\,, \quad \text{ with } a=\lambda,\, b=\pm\sqrt{1-\lambda^2}\,.
\end{equation*}
\item[ii)] 
Let $\sigma\in\ms_{sym}\left(\lambda,\Om\right)$ be a two phase conductivity as in \eqref{2fasisig}, 
such that  there exists a solution $u\in H^1(\Om)$ of \eqref{elliptic} with 
$\nabla u \notin 
L_{loc}^{p_{K^{sym}_\lambda}}(\Om;\R^2)$; then, up to a rotation, $\sigma$ takes the following form
\begin{equation*}
\sigma=\chi_{E_1} \diag(S_1,\lambda^{-1})+
\chi_{E_2} \diag(S_2,\lambda)
\,, \quad \text{ with } \lambda\le S_1, \, S_2 \le \lambda^{-1}\,.
\end{equation*}
\end{itemize}
\end{theorem}
\begin{proof}
i) From Proposition \ref{ALN} it follows that 
$K\le \frac{1+\sqrt{1-\lambda^2}}{\lambda} = K_\lambda$. 
On the other hand, i) of Theorem \ref{thm:sigma-thm} 
yields $K^{min}\geq K_{\lambda}$. 
Lemma \ref{primadisug} implies
$K^{min} = \hat{K} = K_\lambda$, thus yielding $g_i h_i = K^{min}$ in both phases.
Now apply Proposition \ref{optimality}  to conclude that $i)$ holds true. 
ii) 
Again from Proposition \ref{ALN}, Theorem \ref{thm:sigma-thm} and Lemma \ref{primadisug}  
we deduce that   $K^{min}=\hat K =\frac{1}{\lambda}$.  The thesis follows from Proposition \ref{simudiago}.
\end{proof}

\subsection{The explicit formula for $K^{min}$}

Here  we give a  direct formula for $K^{min}$ depending on $\sigma_1$ and $\sigma_2$.  

\begin{proposition}\label{sc}
Let $\sigma_1,\, \sigma_2 \in \M$ be positive definite. 
Denote by $\Sigma_1$ and $\Sigma_2$ the symmetric part of 
$\sigma_1$ and $\sigma_2$ respectively, and by $d_1$ and $d_2$ 
their determinant, 
$$
\Sigma_i := \sigma_i^S \,, \quad d_i:=\det\Sigma_i \,, \quad 
i=1,2 \,.
$$
\no
Then,
\begin{equation}\label{Kmindisigma}
K^{min}=\sqrt{
\frac{m + \sqrt{m^2 - 4d_1d_2}}{2\sqrt{d_1d_2}} \
\frac{n + \sqrt{n^2 -4}}{2}
}\,,
\end{equation}
where
\begin{align*}
& m:=(\sigma_2)_{11}  (\sigma_1)_{22}  +
(\sigma_1)_{11}  (\sigma_2)_{22}   -
\frac{1}{2} 
\Big((\sigma_2)_{12}+(\sigma_2)_{21}\Big)
\Big((\sigma_1)_{12}+(\sigma_1)_{21}\Big) \,, \\
& n:= \frac{1}{\sqrt{d_1 d_2}}
\left[
\det\sigma_1 + \det\sigma_2 -\frac{1}{2} 
\Big((\sigma_1)_{21} - (\sigma_1)_{12}\Big)
\Big((\sigma_2)_{21} - (\sigma_2)_{12}\Big)
\right] \,.
\end{align*}
If in addition $\sigma_1,\, \sigma_2 \in \Msim$, then 
\eqref{Kmindisigma} reduces itself to
$$
K^{min} =  
\max \left\{ \sqrt{\frac{1}{\lambda_1}},\sqrt{\lambda_2} \right\}\,,
$$
where $\lambda_1\le \lambda_2$ are the  eigenvalues of 
$\sigma_1^{-1/2} \sigma_2 \sigma_1^{-1/2}$.
\end{proposition}

\begin{proof}
From Proposition \ref{Kmin} it follows that $K^{min}= \sqrt{g_2 h_2}$ where
$g_2$ and $h_2$ are the maximum eigenvalues of 
$\widetilde G_2:=G_1^{-1/2} G_2 G_1^{-1/2}$ and 
$\widetilde H_2:= H_1^{-1/2} H_2 H_1^{-1/2}$
respectively. 
Since by \eqref{ghsigma}, 
$\dsp G_i = \frac{1}{\sqrt{d_i}}\, \Adj \Sigma_i $, one has
$$
\widetilde G_2 = \frac{\sqrt{d_1}}{\sqrt{d_2}} \,
J \, \Sigma_1^{-1/2}\Sigma_2 \Sigma_1^{-1/2} J^T \,.
$$
The eigenvalues of $\widetilde G_2$ are solutions to the following 
equation in $\lambda$
$$
\det\Big( \sqrt{\frac{d_1}{d_2}} \Sigma_2 - 
\lambda \Sigma_1 \Big) = 0 \,.
$$
Set $M:=\Sigma_2\ \Adj\Sigma_1$. Since
 $$
\det\Big( \sqrt{\frac{d_1}{d_2}} \Sigma_2 - 
\lambda \Sigma_1 \Big) = 0 \ \ 
\Longleftrightarrow \ \
\lambda^2 - 
\frac{\tr M}{\sqrt{d_1d_2}} \ \lambda 
+1=0 \,,
$$
the maximum eigenvalue $g_2$ is defined by
$$
g_2 = \frac{\frac{\tr M}{\sqrt{d_1d_2}} + \sqrt{\frac{(\tr M)^2}{d_1 d_2}-4}}{2}\,.
$$
A straightforward computation shows that 
$$
\tr M = 
(\sigma_2)_{11}  (\sigma_1)_{22}  +
(\sigma_1)_{11}  (\sigma_2)_{22}   -
\frac{1}{2} 
\Big((\sigma_2)_{12}+(\sigma_2)_{21}\Big)
\Big((\sigma_1)_{12}+(\sigma_1)_{21}\Big)
=: m \,.
$$
Similarly, one finds that $h_2$ is the largest root of the equation 
$$
\det(H_2 - \lambda H_1) = 0 \,.
$$
Therefore
$$
h_2 = \frac{\tr N + \sqrt{(\tr N)^2 -4}}{2} \,,
$$
where $N:= H_2\ \Adj H_1$. 
It is easily checked that 
$$
\tr N = \frac{1}{\sqrt{d_1 d_2}}
\left[
\det\sigma_1 + \det\sigma_2 -\frac{1}{2} 
\Big((\sigma_1)_{21} - (\sigma_1)_{12}\Big)
\Big((\sigma_2)_{21} - (\sigma_2)_{12}\Big)
\right] =: n\,.
$$
Now assume that $\sigma_1, \sigma_2$ are symmetric.
By \eqref{ellisim} we find 
$g_2 h_2 = 
\max \left\{ \frac{1}{\lambda_1},\lambda_2 \right\}$, where 
$\lambda_1\leq\lambda_2$ are the eigenvalues of 
\begin{equation*}
\tilde\sigma_2 :=
\frac{1}{(\widetilde{H}_2)_{22}} \widetilde G_2^{-1} \,.
\end{equation*}
Since by \eqref{ghsigma}, 
$\dsp G_i = \frac{1}{\sqrt{\det\sigma_i}}\, \Adj \sigma_i $, one has
\begin{equation*}
\begin{aligned}
\tilde\sigma_2 & = \frac{1}{(\widetilde{H}_2)_{22}} 
\frac{\sqrt{\det\sigma_2}}{\sqrt{\det\sigma_1}} \,\,
J \sigma_1^{1/2}\sigma_2^{-1}  \sigma_1^{1/2} J^T \\
& = \frac{1}{(\widetilde{H}_2)_{22}} 
\frac{1}{\sqrt{\det(\sigma_1^{1/2}\sigma_2^{-1}  \sigma_1^{1/2})}}
\Adj( \sigma_1^{1/2}\sigma_2^{-1}  \sigma_1^{1/2})\\
& = \frac{1}{(\widetilde{H}_2)_{22}} 
\sqrt{\det(\sigma_1^{1/2}\sigma_2^{-1}  \sigma_1^{1/2})}
\big( \sigma_1^{1/2}\sigma_2^{-1}  \sigma_1^{1/2}\big)^{-1}
 \,.
\end{aligned}
\end{equation*}

\no
The eigenvalues of $\tilde\sigma_2$ are those of 
$\sigma_1^{-1/2}\sigma_2 \sigma_1^{-1/2}$ as soon as we prove that 
$$ 
(\widetilde{H}_2)_{22} =
\sqrt{\det(\sigma_1^{1/2}\sigma_2^{-1}  \sigma_1^{1/2})} =
\sqrt{\frac{\det\sigma_1}{\det\sigma_2}}
\,.
$$
This follows from the fact that $H_1$ and $H_2$ are diagonal and therefore 
$$
(\widetilde{H}_2)_{22} = \frac{({H}_2)_{22}}{({H}_1)_{22}}
= \frac{\sqrt{\det\sigma_1}}{\sqrt{\det\sigma_2}} \,.
$$
\end{proof}

\begin{remark}
{\rm Keeping the notation of Proposition \ref{sc}, 
if $\sigma_1,\, \sigma_2 \in \Msim$ are positive definite, 
a straightforward computation shows that  
$$
p_{K^{min}} = \frac{2}{1-\min\left\{ \sqrt{\frac{1}{\lambda_1}},\sqrt{\lambda_2} \right\}}\,.
$$ 
}
\end{remark}

\section{Some $G$-closure results revisited}
Quasiconformal mappings  appear in many branches of mathematics. Only rather recently  they have shown their power in the theory of composites. In the composite material literature one of the typical goals is to determine the so-called ``$G$-closure of a set of conductivities''. Roughly speaking this means the following. Assume that two matrices, called the conductivity of the ``phases'' and denoted by $\sigma_1,\sigma_2\in \msl$ are given. Consider a two phase composites, i.e. a conductivity $\sigma$ of the form  $\sigma= \sigma_1 \chi_{E_1}+\chi_{E_2}\sigma_2$ where $E_1$ and 
$E_2$ are a pair of disjoint measurable sets with $E_1\cup E_2=\Om$. 
The task is to find the set of all possible ``effective'' tensors $\sigma^*$ 
that can be obtained by mixing these two phases while letting $E_1$ and 
$E_2$ vary in all the admissible ways. To make this concept precise, one needs to define an appropriate concept which is called $H$-convergence and was invented by Murat and Tartar. This notion was a  general framework which was necessary to treat the case non-symmetric conductivity $\sigma$ which could not be treated by the $G$-convergence previously introduced by Spagnolo. In both cases one can establish compactness results and  a notion of closure. We will continue to call it $G$-closure according to tradition even if, in this particular case, one really needs to use the $H$-convergence because the tensor $\sigma$ is not assumed to be symmetric a priori. We refer to the recent book of Tartar \cite{T} and reference therein for an extensive treatment.

In this context, an  extensive use of certain special properties of solutions to \eqref{elliptic} and therefore to \eqref{stream}, has been made. For an accurate review, we refer to \cite{Mbook}, see Chapter 4. As a particularly interesting case, we consider
Milton's work computing the so called  $G$-closure of a mixture of two materials with arbitrary volume fractions \cite{M-hall}. 
In the symmetric case, i.e. when both phases have a symmetric conductivity, the $G$-closure was found in the eighties. The result has a long history which is reviewed in a very recent work by Francfort and Murat \cite{FM}. We refer the reader to the reference therein for more details about the original work.

Milton studied the general case without assuming symmetry. He proved that one can recover the $G$-closure for this case by first reducing the problem to the study of a two-phase composite in which, in addition,  each  phase is  symmetric, \cite{M-hall} and Chapter 4.3 in \cite{Mbook},  and then applying the results for the symmetric case.
  Milton explained how his work  was  generalizing  previous work by many authors including Keller, Dykhne, Mendelsohn and that, in turn, he was inspired by some work of Francfort and Murat and some unpublished work by Tartar now available in \cite{T}, Lemma 20.3: in two dimensions ``homogenization commutes with certain Moebius transformations''.  Without entering into too many details, we want to emphasize here that the basic ingredients behind these transformations have an elegant geometrical counterpart when expressed in terms of the Beltrami equation.

  When $\sigma$ is two-phase, by \eqref{ghsigma}, so are the matrices $H$ and $G$. In particular $H=H_1\chi_1+H_2\chi_2$. Consider now the equation \eqref{beltramiGH} and make the affine change of variable $f\to F=A f$, then $F$ satisfies a new equation in which  the matrix $H$ is replaced by $H_{A}:=A^T HA/(\det A)$. Therefore choosing $A=H^{-\frac 1 2} R_2^T$
  with $R_2\in SO(2)$ and such that $R^T_2 H_2 R=:D_2$ is diagonal, one has
  \begin{equation}
  H_{A}= I \chi_1+D_2\chi_2
  \end{equation}
  so that $H_{A}$ is diagonal and thus  $(H_A)_{12}$ is identically zero.
This in turn implies, by \eqref{sigma(G,H)bis} that the corresponding conductivity $$\sigma_{A}:=\frac{G^{-1}+(H_A)_{12} J}{(H_A)_{22}}$$ is symmetric.

  We observe, in passing, that applying the same strategy to the domain of $f$ one can independently reduce a two-phase $G$ to the form
  \begin{equation}
  G_{B}= I \chi_1+G_2\chi_2
  \end{equation}
  with $G_2$ a diagonal matrix by a linear transformation $x\to Bx$.

In the work of Milton, the ``symmetrization''  property for a two-phase composites is obtained as  follows.
Let $\lambda\in[0,1)]$ and let $\sigma\in \msl$. 
Set
\begin{equation}\label{A,S}
A=
\left(
\begin{array}{ll}
a&b\\
c&d
\end{array}
\right),
 \quad \Sigma_A=(a\sigma+bJ)(cI+d J \sigma)^{-1}
\end{equation}
and let $U_{\sigma}=(u_{\sigma}^1,u^2_{\sigma})$ be any solution to the equation \eqref{stream} i.e. $\sigma\nabla u_{\sigma}^1=J^T \nabla u^2_{\sigma}$. 

\begin{proposition}\label{K-T}
For any two-phase composites, there exists $A$  as in \eqref{A,S} such that the corresponding $\Sigma_A$ is symmetric and moreover for some $\lambda^\prime\in [0,1)$ one has $\Sigma_A\in \mslp$.
 \end{proposition}
 
 To continue the argument Milton needs to prove that the $G$-closure problem relative to $\Sigma_A$ is mapped one to one into that relative to $\sigma$. He uses \cite{M-hall} the commutation of the linear fractional transformation $\sigma\to \Sigma_A$ with homogenization, see also \cite{T}, Lemma  20.3.
 
 Our perspective is to use  the following property.

\begin{proposition}\label{prop2}
For any given $A$ as in \eqref{A,S} for which $\Sigma_A\in \mslp$ for some $\lambda^\prime\in [0,1)$,
 there exists 
\begin{equation}
A^\prime=
\left(
\begin{array}{ll}
a^\prime&b^\prime\\
c^\prime&d^\prime
\end{array}
\right),
\end{equation}
such that any solution $U_{\Sigma_A}=(u_{\Sigma_A}^1,u^2_{\Sigma_A})$ to
 $\Sigma_{A}\nabla u_{\Sigma_A}^1=J^T \nabla u^2_{\Sigma_A}$ takes the form
\begin{equation}
U_{\Sigma_A}= A^\prime U_{\sigma}.
 \end{equation}
 \end{proposition}
 
  \begin{proof}
We need to prove that there exist $\{a^\prime, b^\prime,c^\prime,d^\prime\}$ such that
$$
\Sigma_A (a^\prime \nabla u^1_{\sigma}+b^\prime \nabla u^2_{\sigma})=J^T(c^\prime \nabla u^1_{\sigma}+d^\prime \nabla u^2_{\sigma}),
$$
which is equivalent to show that
$$
 (a^\prime \Sigma_A -c^\prime J^T)\nabla u^1_{\sigma}+
 (b^\prime \Sigma_A -d^\prime J^T)\nabla u^2_{\sigma}=0.
 $$
 We now use the equation  $ \sigma \nabla u^1_{\sigma}=J^T \nabla u^2_{\sigma}$ and write the previous equation as
$$
 [a^\prime \Sigma_A -c^\prime J^T+
 (b^\prime \Sigma_A -d^\prime J^T)J \sigma ]\nabla u^1_{\sigma}=0.
 $$
 One possible solution (actually the only one) is found if the matrix in square brackets is zero i.e. if and only if
 $$
 a^\prime \Sigma_A -c^\prime J^T+
 (b^\prime \Sigma_A -d^\prime J^T)J \sigma =0 \Leftrightarrow
 \Sigma_A(a^\prime I +b^\prime J \sigma)=c^\prime J^T+d^\prime \sigma \Leftrightarrow
 $$
 $$
   \Sigma_A=(c^\prime J^T+d^\prime \sigma)(a^\prime I +b^\prime J \sigma)^{-1}
 $$
 and the latter is equivalent to make the following choice:
\begin{equation}\label{Ap}
A^\prime=
\left(
\begin{array}{cc}
c&d\\
-b&a 
\end{array}
\right).
\end{equation}
  \end{proof}
Proposition  \ref{prop2} is the key property to the commuting rule and it is, indeed, a linear change of variables in the target space of the underlying quasiregular mapping $U=(u,v)$, solution to \eqref{stream}.

Finally one may wonder whether \eqref{Ap} can be chosen in such a way to have $\Sigma_A \in \mslp$ for some $\lambda^\prime>0$. To check this we first note that
\begin{multline*}
 \Sigma_A=(a \sigma +b J)(cI +d J \sigma)^{-1}=\frac{a \sigma +b J}{c^2\det \sigma +d^2}{\rm Adj}(cI +d J \sigma)=
 \\
 \frac{a \sigma +b J}{c^2\det \sigma +d^2} (cI +d  J^T(\sigma^T J^T)J)=
\frac{a \sigma +b J}{c^2\det \sigma +d^2} (cI +d  J^T\sigma^T )=\\
\frac{a c \sigma +b c  J+ a d \sigma J^T\sigma^T+bd\sigma^T}{c^2\det \sigma +d^2}= \frac{a c \sigma +bd\sigma^T+ b c  J+ a d \det \sigma J^T}{c^2\det \sigma +d^2}.
\end{multline*}
It follows that
\begin{equation}
(\Sigma_A)^S=\frac{\Sigma_A+\Sigma_A^T}{2}=
\frac{a c+bd}{c^2\det \sigma +d^2}\sigma^S. 
\end{equation}
Therefore, recalling  \eqref{Ap}, the first necessary condition to \eqref{ellipticity-sigma1} can be expressed as follows
\begin{equation}\label{nec1}
c^2+d^2>0\quad,\quad ac+bd>0\Leftrightarrow ac+bd>0 \Leftrightarrow \det A^{\prime}>0.
\end{equation}
Now we need to consider $\Sigma_A^{-1}$.
\begin{multline*}
\Sigma_A^{-1}=(cI +d J \sigma) (a \sigma +b J)^{-1}=\frac{cI +d J \sigma}{a^2\det \sigma +b^2}{\rm Adj}(a \sigma +b J)=
\\
\frac{cI +d J \sigma}{a^2\det \sigma +b^2}(a J \sigma^T J^T+b J^T)=
\frac{(cI +d J \sigma)(a J \sigma^T J^T+b J^T)}{a^2\det \sigma +b^2}=
\\
\frac{a c J \sigma^T J^T + ad J \sigma J \sigma^T J^T+ b c  J^T +b d J \sigma J^T}{a^2\det \sigma +b^2}=
\\
\frac{a c J \sigma^T J^T +b d J \sigma J^T + ad J^T \det \sigma+ b c  J^T}{a^2\det \sigma +b^2}.
\end{multline*}
It follows that
\begin{equation}
(\Sigma_A^{-1})^S= \frac{a c+bd}{a^2\det \sigma +b^2}J \sigma^SJ^T.
\end{equation}
Therefore the second necessary condition to \eqref{ellipticity-sigma2} is expressed as follows 
\begin{equation}\label{nec2}
a^2+b^2>0\,,\quad ac+bd>0\Longleftrightarrow ac+bd>0
\Longleftrightarrow \det A^{\prime}>0.
\end{equation}
Putting \eqref{nec1} and \eqref{nec2} together we obtain 
\begin{equation}
\Sigma_A \in \mslp \quad \hbox{for some}\quad \lambda^\prime>0 \Longleftrightarrow
\det A^{\prime}>0 \,.
\end{equation}
Again, this fact has a clear interpretation in the language of the Beltrami system, recalling that $A'$ represents a linear change of variables in the target space  and that
ellipticity in this context is measured according to  Proposition \ref{prodotto-autovalori}.

\section*{Acknowledgements}
We thank Graeme Milton for an insightful discussion about this problem.

\end{document}